\documentclass{article}
\usepackage{mathtools}
\usepackage{amsfonts}
\usepackage{amsmath}
\usepackage{amsthm}
\usepackage{amssymb}
\usepackage{diagbox} 
\newtheorem{theorem}{Theorem}
\newtheorem{lemma}[theorem]{Lemma}

\begin{document}
\title{A RECURRENCE RELATION ASSOCIATED WITH UNIT-PRIMITIVE MATRICES}
\author{Byeong-Gil Choe\footnote{Chonnam National University, heakum2@gmail.com}, Hyeong-Kwan Ju\footnote{Chonnam National University, hkju@jnu.ac.kr} \footnote{corresponding author, Department of Mathematics, Chonnam National University, Gwangju, Republic of Korea} }
\date{\today}
\maketitle
\begin{abstract}
In this paper we obtained several properties that the characteristic polynomials of the unit-primitive matrix satisfy. In addition, using these properties we have shown that the recurrence relation given as in the formula \eqref{CP} is true. In fact, Xin and Zhong(\cite{XIN_ZHONG}) showed it earlier. However, we provide simpler method here.  
\end{abstract}

\section{Introduction}

The unit-primitive matrix comes naturally when  computing discrete volumes of certain graph polytopes. Below, the related terms and results will be explained.\\
For a given positive integer $m$, we let $B(m)=(b_{ij})$ $(1\leq i, j\leq m)$ be a square matrix of size $m\times m$ satisfying

$$
b_{ij}= \begin{cases}1&i+j \leq m+1\\0&\text{otherwise}\end{cases}
$$

We call this type of upper triangular matrix a \textbf{unit-primitive matrix} of size $m$. For example, unit-primitive matrix of size 5 and its inverse matrix are as follows.

$$
B(5)=
\begin{pmatrix*}[r]
1&1&1&1&1\\
1&1&1&1&0\\
1&1&1&0&0\\
1&1&0&0&0\\
1&0&0&0&0
\end{pmatrix*}, \quad
B(5)^{-1}=
\begin{pmatrix*}[r]
0&0&0&0&1\\
0&0&0&1&-1\\
0&0&1&-1&0\\
0&1&-1&0&0\\
1&-1&0&0&0
\end{pmatrix*}
$$

If $M$ is a square matrix, we denote the sum of all entries of $M$ by $s(M)$. So, $s(M)=u^t M u$ for the column vector $u$ all of whose entries are 1.
We define a bi-variate sequence $b(n,m)(n, m \geq 0)$ as follows.
$$ \quad b(n,m)= \begin{cases}1&\text{$n$=0 or $m$=0}\\s\big( B(m+1)^{n-1} \big) &\text{$n$ and $m$ positive}\end{cases}
$$

\begin{table}
\centering
\begin{tabular}{|c|c c c c c c c c c c c|}
    \hline
    \diagbox{$n$}{$m$} & 0 & 1 & 2 & 3 & 4 & 5 & 6 & 7 & 8 & 9 & $\cdots$ \\ \hline
    0 & 1 & 1 & 1 & 1 & 1 & 1 & 1 & 1 & 1 & 1 & $\cdots$ \\
    1 & 1 & 2 & 3 & 4 & 5 & 6 & 7 & 8 & 9 & 10 & $\cdots$ \\
    2 & 1 & 3 & 6 & 10 & 15 & 21 & 28 & 36 & 45 & 55 & $\cdots$ \\
    3 & 1 & 5 & 14 & 30 & 55 & 91 & 140 & 204 & 285 & 385 & $\cdots$ \\
    4 & 1 & 8 & 31 & 85 & 190 & 371 & 658 & 1086 & 1695 & 2530 & $\cdots$ \\
    5 & 1 & 13 & 70 & 246 & 671 & 1547 & 3164 & 5916 & 10317 & 17017 & $\cdots$ \\
    $\vdots$ & $\vdots$ & $\vdots$ & $\vdots$ & $\vdots$ & $\vdots$ & $\vdots$ & $\vdots$ & $\vdots$ & $\vdots$ & $\vdots$ &  \\ \hline  
\end{tabular}
  \caption{Table of $b(n,m)$}
  \label{Table of $b(n,m)$}
\end{table}

Our main concern is that the following recurrence relation holds for the sequence $(b(n,m))_{n,m\geq 0}$.\\
\begin{equation}\label{CP}
    b(n,m)=b(n,m-1)+\sum_{k\geq 0} b(2k, m-1)b(n-1-2k, m)
\end{equation}
\\
This number appeared in chemistry as the number of Kekul\'{e} structures of the benzenoid hydrocarbons.(See \cite{CY_GUT},  \cite{XIN_ZHONG} for details.). It is also listed with an id A050446 in \cite{OEIS}).
Let $F_m (x)=\sum_{n\geq 0}b(n,m)x^n$, $G_n (x)=\sum_{m\geq 0}b(n,m)x^m$, $Q_m (x)$
=$ \det{(I-xB(m))}$ and $R_m (x)=\det {\begin{pmatrix} 0&u^t\\-u&I-xB(m) \end{pmatrix}}$, where $u=(1,1,\cdots,1)^t \in \mathbb{R}^m$.\\
The sequence $(G_n(y))_{n\geq 0}$ of the Ehrhart series is well analyzed in detail by Xin and Zhong(\cite{XIN_ZHONG}). We want to describe our main results in a slightly different way, however.  Now, let $M^*$ be an adjugate matrix of the square matrix $M$. 
\begin{lemma}
$s(M^*)=\det{\begin{pmatrix} 0&u^t\\-u&M \end{pmatrix}}$ for all matrix M of size $n \times n$. 
\end{lemma}

\begin{proof}

Let $m_i$ be the $i$th column of the matrix $M$, and 

$$M_i= \det (m_1, \cdots ,m_{i-1},u,m_{i+1}, \cdots ,m_n).$$ Then
 \begin{align*}
 s(M^*)=&u^tM^*u=u^t[M_1,M_2,\cdots ,M_n]^t \\
       =& \det (u,m_2,m_3,\cdots ,m_n)  + \det (m_1,u,m_3,\cdots ,m_n) \\ 
        &+\cdots + \det (m_1,m_2,m_3,\cdots ,u) \\
       =& \det (u,m_2,m_3,\cdots ,m_n)- 
          \det (u,m_1,m_3,\cdots ,m_n) \\ 
        & +\cdots +(-1)^{n+1} \det (u,m_1,m_2,\cdots ,m_{n-1}) \\
      = &\ \det {\begin{pmatrix} 0&u^t\\-u&M \end{pmatrix}}.
\end{align*}
\end{proof}
\begin{theorem}
Let $E_m(x)= \sum_{n\geq 0}b(n+1,m)x^n$. Then
$$ E_m(x)=\frac{\det \begin{pmatrix} 0 & u^t \\ -u & I-xB(m+1) \end{pmatrix}}{\det(I_{m+1}-xB(m+1))}=\frac{R_{m+1}(x)}{Q_{m+1}(x)}$$
\end{theorem}
\begin{proof}
\begin{align*}
   E_m(x) & =\sum_{n\geq 0}b(n+1,m)x^n=\sum_{n\geq 0}s(B(m+1)^n)x^n \\
          & =s\left(\sum_{n\geq 0}(xB(m+1))^n\right)=s\left[I-xB(m+1)^{-1} \right] \\
          & =s\left(\frac{(I-xB(m+1))^*}{\det (I-xB(m+1))}\right)=\frac{1}{Q_{m+1}(x)}s((I-xB(m+1))^*) \\
          & =\frac{1}{Q_{m+1}} \left[ u^t(I-xB(m+1))^*u \right] \\
          & =\frac{1}{Q_{m+1}(x)}\det \begin{pmatrix}
     0 & u^t \\
     -u & I-xB(m+1)
 \end{pmatrix}
 =\frac{R_{m+1}(x)}{Q_{m+1}(x)}. 
\end{align*} 
\end{proof}
\section{Properties of $Q_m(x)$}
In this section we list the properties of $Q_m(x)$ and prove them.
\begin{theorem}\label{thm3}
$Q_m(x)=\det (I-xB(m))$ satisfies the following properties.
\begin{align*}
(1)& \quad Q_m(x)=-x Q_{m-1}(-x)+Q_{m-2}(x) \quad (m\geq 2), \text{ and} \\
   & \quad Q_0(x)=1, Q_1(x)=1-x.\\
(2)& \quad Q_m(x)Q_{m+1}(x)-Q_m(-x)Q_{m+1}(-x)=2 \quad (m\geq 0). \\
(3)& \quad Q_{m+1}(x)Q_{m+1}(-x)-Q_{m+2}(x)Q_m(-x)=x \quad (m\geq 0). \\
(4)& \quad Q_m(x)=x R_m(x)=Q_{m-1}(-x) \quad (m\geq 1).\\
\end{align*} 
\end{theorem}
\begin{proof}
    (1) For $m \geq 2$, 
    \begin{align*}
    Q_m(x) & =\begin{vmatrix}
        1-x & -x & -x & \cdots & -x & -x & -x \\
        -x & 1-x & -x & \cdots & -x & -x & 0 \\
        -x & -x & 1-x & \cdots & -x & 0 & 0 \\
        \vdots & \vdots & \vdots & \vdots & \vdots & \vdots & \vdots \\
        -x & -x & 0 & \cdots & 0 & 1 & 0 \\
        -x & 0 & 0 & \cdots & 0 & 0 & 1 \\
    \end{vmatrix} \\\\
       &=\begin{vmatrix}
        1 & -x & -x & \cdots & -x & -x & -x \\
        0 & 1-x & -x & \cdots & -x & -x & 0 \\
        0 & -x & 1-x & \cdots & -x & 0 & 0 \\
        \vdots & \vdots & \vdots & \vdots & \vdots & \vdots & \vdots \\
        0 & -x & 0 & \cdots & 0 & 1 & 0 \\
        0 & 0 & 0 & \cdots & 0 & 0 & 1 \\
    \end{vmatrix}
    +\begin{vmatrix}
        -x & -x & -x & \cdots & -x & -x & -x \\
        -x & 1-x & -x & \cdots & -x & -x & 0 \\
        -x & -x & 1-x & \cdots & -x & 0 & 0 \\
        \vdots & \vdots & \vdots & \vdots & \vdots & \vdots & \vdots \\
        -x & -x & 0 & \cdots & 0 & 1 & 0 \\
        -x & 0 & 0 & \cdots & 0 & 0 & 1 \\
    \end{vmatrix} \\
    & \quad \text{ by the splitting of the first column } \\
    \end{align*} 
    \begin{align*}
    & =\det (I-xB(m-2)) 
    +\begin{vmatrix}
        -x & 0 & 0 & \cdots & 0 & 0 & 0 \\
        -x & 1 & 0 & \cdots & 0 & 0 & x \\
        -x & 0 & 1 & \cdots & 0 & x & x \\
        \vdots & \vdots & \vdots & \vdots & \vdots & \vdots & \vdots \\
        -x & 0 & x & \cdots & x & 1+x & x \\
        -x & x & x & \cdots & x & x & 1+x \\
    \end{vmatrix} \\\\
    & = Q_{m-2}(x)-x
    \begin{vmatrix}
        1 & 0 &  \cdots & 0 & 0 & x \\
        0 & 1 & \cdots & 0 & x & x \\
       \vdots & \vdots & \vdots & \vdots & \vdots & \vdots  \\
        0 & x & \cdots & x & 1+x & x \\
        x & x & \cdots & x & x & 1+x \\
    \end{vmatrix} \\\\  
    & = Q_{m-2}(x)-x
    \begin{vmatrix}
        1+x & x &  \cdots & x & x & x \\
        x & 1+x & \cdots & x & x & 0 \\
       \vdots & \vdots & \vdots & \vdots & \vdots & \vdots \\
        x & x & \cdots & 0 & 1 & 0 \\
        x & 0 & \cdots & 0 & 0 & 1 \\
    \end{vmatrix} \\
     & \quad \text{ in reverse order of rows and columns in the second determinant } \\
    & = Q_{m-2}(x)-xQ_{m-1}(-x) \\
   \end{align*}  
    (2) Use induction on $m$.   
    Formula (2) holds for $m=1$ since
$$Q_0(x)Q_1(x)+Q_0(-x)Q_1(-x)=(1-x)(1-x-x^2)+(1+x)(1+x-x^2)=2$$
     We assume that formula (2) holds for $m \leq k$.
\begin{align*}  
    & Q_{k+1}(x)Q_{k+2}(x)+Q_{k+1}(-x)Q_{k+2}(-x) \\
    & = Q_{k+1}(x)(-xQ_{k+1}(-x)+Q_k(x))+Q_{k+1}(-x)(xQ_{k+1}(x)+Q_k(-x) \\
    & = Q_{k+1}(x)Q_k(x)+Q_{k+1}(-x)Q_k(-x)=2.
\end{align*} 
    \\
    (3) Similar to the previous one, we also use induction on $m.$ 
    Formula (3) holds for $m=0$ since
    $$Q_1(x)Q_1(-x)-Q_2(x)Q_0(-x)=(1-x)(1+x)-(1-x-x^2)(1)=x.$$
    We assume that formula (3) holds for $m \leq k$.
\begin{align*}
    & Q_{k+2}(x)Q_{k+2}(-x)-Q_{k+3}Q_{k+1}(-x) \\
    & = Q_{k+2}(x)(xQ_{k+1}(x)+Q_k(-x))-(-xQ_{k+2}(-x)+Q_{k+1}(x))Q_{k+1}(-x) \\
    & = x \left[ Q_{k+1}(x)Q_{k+2}(x)+Q_{k+1}(-x)Q_{k+2}(-x) \right]+\left[ Q_{k}(-x)Q_{k+2}(x)-Q_{k+1}(x)Q_{k+1}(-x) \right] \\
    & = 2x-x = x, 
\end{align*}
    by the formula (2) and the induction assumption.
\newline \\
    (4)  \begin{align*}
    & Q_m(x)+xR_m(x) \\
    & = \det (I-xB(m))+x\det \begin{pmatrix}
        0 & u^t \\
        -u & I-xB(m)
    \end{pmatrix} \\
    & = \det \begin{pmatrix}
        1 & u^t \\
        -xu & I-xB(m)
    \end{pmatrix} \\
    & = \begin{vmatrix}
    1 & 1 & 1 & 1 & \cdots & 1 & 1 & 1 \\
        -x & 1-x & -x & -x & \cdots & -x & -x & -x \\
        -x & -x & 1-x & -x & \cdots & -x & -x & 0 \\
        -x & -x & -x & 1-x & \cdots & -x & 0 & 0 \\
        \vdots & \vdots & \vdots & \vdots & \vdots & \vdots & \vdots \\
        -x & -x & -x & -x & \cdots & 1 & 0 & 0 \\
        -x & -x & -x & 0 & \cdots & 0 & 1 & 0 \\
        -x & -x & 0 & 0 & \cdots & 0 & 0 & 1 \\
    \end{vmatrix} \\\\
    \end{align*}
    \begin{align*}
    & = \begin{vmatrix}
    1 & 1 & 1 & 1 & \cdots & 1 & 1 & 1 \\
        0 & 1 & 0 & 0 & \cdots & 0 & 0 & 0 \\
        0 & 0 & 1 & 0 & \cdots & 0 & 0 & x \\
        0 & 0 & 0 & 1 & \cdots & 0 & x & x \\
        \vdots & \vdots & \vdots & \vdots & \vdots & \vdots & \vdots \\
        0 & 0 & 0 & 0 & \cdots & 1+x & x & x \\
        0 & 0 & 0 & x & \cdots & x & 1+x & x \\
        0 & 0 & x & x & \cdots & x & x & 1+x \\
    \end{vmatrix} \\\\
    & =  \begin{vmatrix}
        1 & 0 &  \cdots & 0 & 0 & x \\
        0 & 1 & \cdots & 0 & x & x \\
       \vdots & \vdots & \vdots & \vdots & \vdots & \vdots  \\
        0 & 0 & \cdots & 1+x & x & x \\
        0 & x & \cdots & x & 1+x & x \\
        x & x & \cdots & x & x & 1+x \\
    \end{vmatrix} \\\\
    & = \begin{vmatrix}
        1+x & x &  \cdots & x & x & x \\
        x & 1+x & \cdots & x & x & 0 \\
       \vdots & \vdots & \vdots & \vdots & \vdots & \vdots  \\
        x & x & \cdots & 1 & 0 & 0 \\
        x & x & \cdots & 0 & 1 & 0 \\
        x & 0 & \cdots & 0 & 0 & 1 \\
    \end{vmatrix} \\
    & = Q_{m-1}(-x) \\
  \end{align*}  
\end{proof}

Here is another interesting property on $(Q_m(x))_{m\geq 0}$. For reference its ordinary generating function is given below. (See \cite{HKJU} for the proof.)
\begin{theorem}
    $$\sum_{m \geq 0} Q_m(x)t^m = \frac{(1+t)(1-t^2-xt)}{(1-t^2)^2+(xt)^2}.$$
\end{theorem}

\section{Recurrence Relation of the Sequence $b(n,m)$}

Note that $F_m(x)=1+xE_m(x).$
By the property (4) of Theorem\ref{thm3}, we obtain the generating function $F_m(x)$ of the sequence $\left( b(n,m) \right )_{n\geq 0}.$
\begin{theorem}    
$$F_m(x)=\frac{Q_m(-x)}{Q_{m+1}(x)}.$$
\end{theorem}
\begin{proof}
 \begin{align*}
  F_m(x) &= 1+xE_m(x) 
        =  1+x \frac{R_{m+1}(x)}{Q_{m+1}(x)} \\
        &=  \frac{Q_{m+1}+xR_{m+1}(x)}{Q_{m+1}(x)} 
         =  \frac{Q_m(-x)}{Q_{m+1}(x)}. 
 \end{align*} 
\end{proof}
Similar to the Chebyshev function of the second kind, $F_m(x)$ has an expression by a trigonometric function as in the next theorem.(For details, see references \cite{HKJU} and \cite{XIN_ZHONG}.)
\begin{theorem}
For each positive integer $m,$ 
$$F_m(x)=(-1)^{m+1} \frac{\cos (\frac{2m+1}{2})\theta}{\cos (\frac{2m+3}{2})\theta} = 
\frac{\sin(m+1)\theta -\sin m\theta}{\sin(m+2)\theta -\sin(m+1)\theta},$$
where $\theta= \cos^{-1}\Big( \frac{(-1)^m x}{2} \Big).$
\end{theorem}
\begin{theorem}
    Generating function $F_m(x)$ satisfies the continued fraction property:
    $$F_m(x)=\frac{1}{-x+F_{m-1}(-x)},\ F_0(x)=\frac{1}{1-x}.$$
 \end{theorem}   
     First three formulas are listed here:
 \begin{align*}    
    &F_1(x)=\frac{1}{-x+F_0(-x)}=\frac{1}{-x+\frac{1}{x+1}}=\frac{1+x}{1-x-x^2}=\frac{Q_1(-x)}{Q_2(x)}\\
    &F_2(x)=\frac{1}{-x+F_1(-x)}=\frac{1+x-x^2}{1-2x-x^2+x^3}=\frac{Q_2(-x)}{Q_3(x)}\\
    &F_3(x)=\frac{1}{-x+F_2(-x)}=\frac{1+2x-x^2-x^3}{1-2x-3x^2+x^3+x^4}=\frac{Q_3(-x)}{Q_4(x)}\\
\end{align*}
\begin{proof}
    For a positive integer $m\geq 2$, if we use the property (1) of Theorem\ref{thm3}. we obtain the following:\\
\begin{align*}
   F_m(x)&=\frac{Q_m(-x)}{Q_{m+1}(x)}=\frac{Q_m(-x)}{-xQ_m(-x)+Q_{m-1}(x)}\\
    &=\frac{1}{-x+\frac{G_{m-1}(x)}{Q_m(-x)}}=\frac{1}{-x+F_{m-1}(-x)}. 
\end{align*}    
\end{proof}
\begin{theorem}
    Let the sequence $\{ c(n,m) \}_{n,m\geq 0}$ satisfies the recurrence relation stated as below:
    \begin{equation}\label{recursive1}
        c(n,m)=c(n,m-1)+\sum_{k\geq 0}c(2k, m-1) c(n-1-2k,m)(c(n,0)=1, \forall n \geq 0)
    \end{equation}
Then $c(n,m)=b(n,m)$ for all nonnegative integers $n$, $m$.
\end{theorem}
 In fact, this has been proved by Xin and Zhong(\cite{XIN_ZHONG}). However, we proved it another way by using properties (2) and (3) of Theorem\ref{thm3} as below. Our proof seems to be short and simple compared to their lengthy proof which amounts to several pages.
\begin{proof}
    Let $$C_m(x)=\sum_{n\geq 0}c(n,m)x^n.$$ The proof is done if we show that $C_m(x)=F_m(x)$ for all $m\geq 0$. Note that from the recurrence relation (\ref{recursive1}) we get the following formula:
    \begin{equation}\label{recursive2}
        C_m(x)=C_{m-1}(x)+xC^e _{m-1}(x)C_m(x) \text{ with }C_0(x)=\frac{1}{1-x},
    \end{equation}
    where $$C^e_m(x)=\frac{1}{2}(C_m(x)+C_m(-x)).$$ 
    From the equation (\ref{recursive2}), we obtain the following:
    \begin{equation}\label{genftn1}
        C_m(x)=\frac{C_{m-1}(x)}{1- x(C_{m-1}(x)+C_{m-1}(-x))/2}
    \end{equation}

    We will use mathematical induction on $m$ to prove that $C_m(x)=F_m(x)$ for all $m \geq 0$.
    It is obvious that $C_0(x)=\frac{1}{1-x}=F_0(x)$.
    We assume that $C_i(x)=F_i(x)$ for all $i\leq m-1$.
    By the assumption, from the formula(\ref{genftn1}) , we have
    \begin{align*}\label{genftn2}
        C_m(x)&=\frac{F_{m-1}(x)}{1-x(F_{m-1}(x)+F_{m-1}(-x))/2} \\
              &=\frac{\frac{Q_{m-1}(-x)}{Q_m(x)}}{1-x\left( \frac{Q_{m-1}(-x)}{Q_m(x)}+\frac{Q_{m-1}(x)}{Q_m(-x)} \right)/2} \\
             &=\frac{\frac{Q_{m-1}(-x)}{Q_m(x)}}{1-\frac{x}{2}\frac{Q_{m-1}(-x)Q_m(-x)+Q_{m-1}(x)Q_m(x)}{Q_m(x)Q_m(-x)}} \\
             &=\frac{\frac{Q_{m-1}(-x)}{Q_m(x)}}{1-\frac{x}{Q_m(x)Q_m(-x)}}=\frac{Q_{m-1}(-x)Q_m(-x)}{Q_m(x)Q_m(-x)-x}=\frac{Q_{m-1}(-x)Q_m(-x)}{Q_{m+1}(x)Q_{m-1}(-x)} \\
             &=\frac{Q_m(-x)}{Q_{m+1}(x)}=F_m(x).\\
   \end{align*}
    In the computation above, the properties (2) and (3) of the Theorem\ref{thm3} were used.
\end{proof}
\section{Concluding Remarks}
As we mentioned earlier, Xin and Zhong(\cite{XIN_ZHONG}) analyzed the generating function $G_n (y)=\sum_{m\geq 0}b(n,m)y^m$ in detail. It is an Ehrhart series of graph polytope for the linear graph which is the rational function of the form $$G_n(y)=\frac{H_n(y)}{(1-y)^{n+1}},$$
where $H_n(y)$ is a polynomial of degree at most $n.$ Consider $$K(x,y)= \sum_{n\geq 0} G_n(y)x^n= \sum_{m\geq 0} F_m(x)y^m. $$
Our question is the following:
\textbf{ what is the form of the generating function $K(x,y)$ exactly?}\\\\
This bi-variate generating function $K(x,y)$ requires more analysis and understanding of us.

\end{document}